\documentclass[12pt]{amsart}

\usepackage{amsmath, amsthm, amssymb}
\usepackage{url} 
\usepackage{graphicx}
\usepackage{xcolor}
\usepackage{moreverb}

\usepackage{fancyvrb}

\def\r{\mathbb R}

\usepackage{listings}
 \lstnewenvironment{code}[1][]
  {\lstset{#1}}
  {}

\newtheorem{theorem}{Theorem}[section]
 \newtheorem{proposition}[theorem]{Proposition}
 \newtheorem{corollary}[theorem]{Corollary}
\theoremstyle{definition}
\newtheorem{definition}[theorem]{Definition}

\newtheorem{remark}[theorem]{Remark}

\begin{document}

\title{Rotational $K^\alpha$-Translators in Minkowski space}

\author{Muhittin Evren Aydin}
\address{Department of Mathematics. Faculty of Science, Firat University, Elazig,  23200 Turkey}
\email{meaydin@firat.edu.tr}
\author{ Rafael L\'opez}
 \address{Departamento de Geometr\'{\i}a y Topolog\'{\i}a\\  Universidad de Granada\\
 18071 Granada, Spain}
 \email{rcamino@ugr.es}

\begin{abstract}
A spacelike surface in Minkowski space $\r_1^3$ is called a $K^\alpha$-translator of the flow by the powers of Gauss curvature if satisfies $K^\alpha= \langle N,\vec{v}\rangle$, $\alpha \neq 0$, where $K$ is the Gauss curvature, $N$ is the unit normal vector field and $\vec{v}$ is a  direction of $\r_1^3$. In this paper, we classify all rotational $K^\alpha$-translators. This classification will depend on the causal character of the rotation axis.   Although the theory of the $K^\alpha$-flow holds for spacelike surfaces, the equation describing $K^\alpha$-translators is still valid for timelike surfaces. So we also investigate the    timelike rotational surfaces that satisfy the same prescribing Gauss curvature equation. 
\end{abstract}
\maketitle

\section{Introduction and formulation of the problems}	

In this paper we consider the translating solitons of the flow by powers of the Gauss curvature $K$ in the Minkowski space $\r_1^3$. Let $X:\Sigma\to\r_1^3$ be a smooth immersion of a spacelike surface $\Sigma$ with positive Gauss curvature $K$. The $K^\alpha$-flow is a one-parameter family of smooth spacelike immersions $X_t=X(\cdot,t)\colon\Sigma\to \r_1^3$, $t\in [0,T)$ such that $X_0=X$ and satisfying   
$$\frac{\partial}{\partial t}X(p,t)=-K(p,t)^\alpha N(p,t),\quad (p,t)\in \Sigma\times [0,T),$$
where $\alpha\not=0$ is a constant, $N(p,t)$ is the unit normal of $X(p,t)$ and $K(p,t)$ is the Gauss curvature at $X(p,t)$.  Our interest are those surfaces $\Sigma$ that move under the flow along a direction $\vec{v}$ of $\r_1^3$. 

\begin{definition} \label{def-1}
Let $\vec{v}\in\r_1^3$. A spacelike surface $\Sigma$ is called a translator by the $K^\alpha$-flow and speed $\vec{v}$ if
\begin{equation}\label{k1}
K^\alpha= \langle N,\vec{v}\rangle.
\end{equation}
\end{definition}
Here we understand that $K>0$ discarding the trivial case  $K=0$.  We also observe that the left-hand side of \eqref{k1} is positive. This implies that if we reverse the orientation $N$, which does not affect to the term $K^\alpha$, then the speed $\vec{v}$ must change of sign. The flow by powers of the Gauss curvature in Euclidean space was initiated by Andrews, Chow and Urbas \cite{an,ch,ur,ur1,ur3}.   In contrast, its study in Minkowski space has received less attention  \cite{an2,gm}. We point out \cite{ju} where has been investigated the Dirichlet problem associated to \eqref{k1}. 
 
In this paper we will classify all $K^\alpha$-translators of $\r_1^3$ of rotational type. A first purpose is  to enrich the theory of the flow by powers of the Gauss curvature in a natural class of surfaces as the surfaces of revolution. Also rotational $K^\alpha$-translators  can be employed as barriers in order to obtain $C^1$ and $C^2$ apriori estimates in the Dirichlet problem associated to the equation \eqref{k1}, such as it occurs  for the prescribing Gauss curvature equation $K=\mbox{ct}$ \cite{bs,bjs,hjs}. On the other hand, the family of surfaces of revolution of $\r_1^3$ is greater than of  Euclidean case because in the Lorentzian ambient space we have three types of surfaces depending on the causal character of the rotation axis. Together this,  the speed $\vec{v}$ has again   three possible causal choices.  

Finally, one more, and no less interesting, observation. If we now leave in mind the  flow theory, it is natural to ask about the solutions of \eqref{k1}   in the case that the surfaces are timelike. Let us point out that  all terms that appear in \eqref{k1} are still valid and  \eqref{k1} can be viewed as a prescribing curvature equation.  Overusing of the language, and for simplicity in the statements, we will say that a timelike surface satisfying \eqref{k1} is a \textit{timelike $K^\alpha$-translator}, leaving $K^\alpha$-translator in case that the surface is spacelike. 

This paper is organized as follows. In Section \ref{sec2} we recall some basics of local differential geometry of surfaces in $\r_1^3$. In separated cases according to the type of surface of revolution, in Section \ref{spac} we classify the rotational $K^{\alpha}$-translators  (Theorems \ref{clas-1}, \ref{clas-2}, \ref{clas-3}).  As a first step, we prove that the speed $\vec{v}$ is not arbitrary: although initially there is not a priori relation between $\vec{v}$ and the direction of the rotation axis of the surface, we will prove that the rotation axis must be parallel to the speed $\vec{v}$  (Propositions \ref{p-tim1}, \ref{p-spa1}, \ref{p-light}). In Section \ref{time}, we do an analogous study for the rotational timelike $K^\alpha$-translators.  

As a consequence of our results, we prove  the non-existence of rotational $K^{\alpha}$-translators orthogonally intersecting the rotation axis when the rotation axis is timelike (Corollary \ref{exist-1}) or spacelike (Corollary \ref{exist-12}). This contrasts with the result in Euclidean setting where, for any $\alpha$, such translators exist (\cite{al}). In spite of this, as long as $\alpha \in (0,1/2)$, there are  rotational $K^{\alpha}$-translators that intersect the rotation axis at a conical point (Corollaries \ref{exist-1}, \ref{exist-2}). A last observation is that in some situations, we will have a $K^\alpha$-translator that can be extended by means of a timelike $K^\alpha$-translator with the same rotation axis. This is because the maximal domain of the function $f(r)$ that determines the generating curve is a bounded interval where at the endpoint  the surface leaves to be spacelike. We will prove that under this circumstance, it is possible to extend the surface beyond of the domain of $f$ with  a rotational  timelike  $K^\alpha$-translator with the same axis.

\section{Preliminaries} \label{sec2}
The Lorentz-Minkowski space $\r_1^3$ is the real $3$-dimensional vector space $\r^3$ endowed with the  metric  $\langle,\rangle=dx^2+dy^2-dz^2$, where
  $(x, y, z)$ are the canonical coordinates in $\mathbb{R}^3$.    Let $\Sigma$ be a surface isometrically immersed in $\r_1^3$ whose induced metric is non-degenerate. If the metric, denoted by $\langle,\rangle$ again, is Riemannian (resp. Lorentzian), the surface is said to be {\it spacelike} (resp. {\it timelike}).  
Denote  by $\nabla^0$ and $\nabla$ the Levi-Civita connection of $\r_1^3$ and the induced connection on $\Sigma$. The second fundamental form $\sigma$ is defined by  the Gauss formula
\begin{equation}\label{3-gf}
\nabla_U^0 V=\nabla_U V+\sigma(U,V),
\end{equation}
for any tangent vector fields $U$ and $V$ on $\Sigma$.  Let $N$ be a (local) unit normal vector field on $\Sigma$. We know that if the immersion is spacelike, the surface is always orientable and that  $\langle N,N\rangle=-1$. If $\Sigma$ is timelike, then $N$ is spacelike and it is only a local unit normal vector field.   To simplify the next notation, let $\langle N,N\rangle=\epsilon\in\{-1,1\}$.     Because $\langle N,N\rangle$ is constant, we have $\langle\nabla_U^0N,N\rangle=0$. Then $\nabla^0_UN$ is tangent to $\Sigma$ and the Weingarten endomorphism $A$ is defined by
$-\nabla^0_UN=A U$. We have
\[
\sigma(U,V)=\epsilon\, \langle\sigma(U,V),N\rangle N=\epsilon\, \langle AU,V\rangle N.
\]
Now \eqref{3-gf} writes as
$$\nabla^0_UV=\nabla_UV+\epsilon\,\langle AU,V\rangle N.$$

The Weingarten map $A$ is self-adjoint  with respect to the induced metric $\langle,\rangle$. In particular,  if $\Sigma$ is spacelike then $A$ is real diagonalizable because   the metric is Riemannian.  The     principal curvatures $\lambda_i$ are the eigenvalues of $A$. If $\Sigma$ is timelike, then $A$ may be or not (real) diagonalizable. We define the (intrinsic) Gauss curvature $K$ of the surface as
\begin{equation*}
K=\epsilon\,\mbox{ det}(A). 
\end{equation*}
If $A$ is diagonalizable, then   $K=\epsilon\, \lambda_1 \lambda_2$. We need the expression of $K$ in local coordinates. 
For a parametrization $X=X(u,v)$ of a (spacelike or timelike) surface, we have
$$K=-\frac{\mbox{det}(X_u,X_v,X_{uu})\mbox{det}(X_u,X_v,X_{vv})-
\mbox{det}(X_u,X_v,X_{uv})^2}{(EG-F^2)^2},$$
where, as usually, 
$$E=\langle X_u,X_u\rangle,\,  F=\langle X_u,X_v\rangle, \, G=\langle X_v,X_v\rangle$$
  are the coefficients of the first fundamental form. Recall that $EG-F^2>0$ if $\Sigma$ is spacelike and $EG-F^2<0$ if $\Sigma$ is timelike.

We now describe the parametrizations of the surfaces of revolution of $\r_1^3$.  In $\r_1^3$ there are three types of surfaces of revolution   according to  the causal character of the rotation axis $L$. Let $B=\{e_1, e_2, e_3\}=\{(1,0,0),(0,1,0),(0,0,1)\}$ be the usual basis of $\r^3$. After a rigid motion, we can assume that $L=\mbox{sp}\{e_3\}$ (timelike), $L=\mbox{sp}\{e_1\}$ (spacelike) or $L=\mbox{sp}\{e_2+e_3\}$ (lightlike).
\begin{enumerate}
\item The axis is timelike, $L=\mbox{sp}\{e_3\}$. The generating curve is $r\mapsto (r,0,f(r))$, $r>0$, where $f$ is a smooth function and
 the parametrization of the surface is
\begin{equation}\label{p1}
X(r,\theta)=( r\cos\theta,r\sin\theta,f(r)).
\end{equation}
Here $EG-F^2=r^2(1-f'^2)$ so the surface is spacelike (resp. timelike) if $f'^2<1$ (resp. $f'^2>1$).
\item The axis is spacelike, $L=\mbox{sp}\{e_1\}$.    The generating curve can be included in the $xz$-plane or in the $xy$-plane. In the first case, if the curve is $r\mapsto (f(r),0,r)$ $r>0$,  the parametrization of the surface is  
\begin{equation}\label{p21}
X(r,\theta)= (f(r),r\sinh \theta,r\cosh \theta).
\end{equation}
Now $EG-F^2=r^2(f'^2-1)$ and the surface is spacelike (resp. timelike) if $f'^2>1$ (resp. $f'^2<1$).
In the second case, if $r\mapsto (f(r),r,0)$, $r>0$, then 
\begin{equation}\label{p22}
X(r,\theta)= (f(r),r\cosh \theta,r\sinh \theta).
\end{equation}
This surface is always timelike because $EG-F^2=-r^2(1+f'^2)<0$.
\item The axis is lightlike, $L=\mbox{sp}\{e_2+e_3\}$.   The generating curve is $r\mapsto (0,f(r)+r,f(r)-r)$, $r>0$,  and the surface is
\begin{equation}\label{p3}
X(r,t)=(2r t, f(r)+r-rt^2,f(r)-r-rt^2).
\end{equation}
We have $EG-F^2=16r^2f'$, so the surface is spacelike (resp. timelike) if $f'>0$ (resp. $f'<0$).
 \end{enumerate}

\section{$K^\alpha$-translators of rotational type} \label{spac}
 In this section, we investigate the $K^\alpha$-translators which are  surfaces of revolution. Since there are three types of such surfaces, our investigation will be into three separated subsections where the rotation axis is timelike, spacelike and lightlike, respectively.
 
 \subsection{The axis is timelike }  \label{ss1}
First we prove that the speed   $\vec{v}$ must be parallel to the rotation axis.

\begin{proposition} \label{p-tim1}
If $\Sigma $ is a rotational $K^{\alpha }$-translator with   timelike axis, then $\vec{v}$ is parallel to its
rotation axis.
\end{proposition}

\begin{proof}
By applying a rigid motion of $\r_1^3$, we take the rotation
axis as the $z$-axis.    The  parametrization of the surface is given in \eqref{p1} with $f'^2<1$ and the unit normal vector field  is%
\begin{equation}\label{n-timelike}
N=-\frac{1}{\sqrt{1-f'^2}}(\cos \theta f',\sin \theta
f',1).
\end{equation}%
The Gauss curvature   is%
\begin{equation*}
K=-\frac{f'f''}{r(1-f'^2)^2}.
\end{equation*}
If $\vec{v}=(v_1,v_2,v_3)$,  equation \eqref{k1} becomes%
\begin{equation*}
P_0(r)+P_1(r)\cos \theta +P_2(r)\sin \theta =0,
\end{equation*}%
where 
\begin{eqnarray*}
P_0 &=&\left( -\frac{f'f''}{r(1-f'^2)^2}\right) ^{\alpha }-v_3(1-f'^2)^{-1/2}, \\
P_1 &=&v_1f'(1-f'^2)^{-1/2},\\
P_2&=&v_2f'(1-f'^2)^{-1/2}.
\end{eqnarray*}%
  Since the functions $%
\left\{ 1,\cos \theta ,\sin \theta \right\} $ are linearly independent,  all coefficients $P_k$ must vanish.    Assume first that $f'\neq 0$ at some point. Then $P_1=0$ and $P_2=0$ imply    $v_1=v_2=0,$ proving the result. If $f'=0$ for all $r$, then $f$ is a constant function and the surface is a plane and $K=0$, which it is impossible.    \end{proof}

After this proposition, the speed is   $\vec{v}=(0,0,v_3)$. According to $N$ in \eqref{n-timelike}, and because $\langle N,\vec{v}\rangle$ is positive in \eqref{k1}, without loss of generality  we can assume that $\vec{v}=(0,0,1)$. Then \eqref{k1} is equivalent to $P_0=0$, that is,
\begin{equation} \label{k2}
\left( -\frac{f'f''}{r(1-f'^2)^2}\right) ^{\alpha }=\frac{1}{\sqrt{1-f'^2}}.
\end{equation}
Let us introduce the auxiliary function
$$h=\frac{1}{\sqrt{1-f'^2}}.$$
In terms of  $h$, equation \eqref{k2} writes as
$$\left(-\frac{hh'}{r}\right)^\alpha=h,$$
or equivalently, 
\begin{equation*}
h^{\frac{\alpha -1}{\alpha}}h'=-r.
\end{equation*}%
 A first  integration gives
\begin{equation*}
h(r)=\left\{
\begin{array}{lll}
me^{\frac{-r^2}{2}},m>0, &  &\alpha =\frac{1}{2} \\
 \left( m-\frac{2\alpha-1 }{2\alpha }r^{2}\right) ^{\frac{\alpha }{2\alpha
-1}},m\in \r, & &\alpha \neq \frac{1}{2}.%
\end{array}%
\right.
\end{equation*}
From $h$, we deduce the value of $f'$, 
  \begin{equation*}
f'(r)=\left\{
\begin{array}{ll}
\pm\left(1-\frac{1}{m^2}e^{r^2 }\right)^{1/2},  & \alpha =\frac{1}{2} \\
\pm\left(1-\left(m-\frac{2 \alpha -1}{2 \alpha }r^2\right)^{\frac{2 \alpha }{1-2 \alpha }}\right)^{1/2},& \alpha \neq \frac{1}{2}.%
\end{array}%
\right.
\end{equation*}%
Here some constraints  arise naturally which determine the domain of $f$. Consequently, we have the classification of rotational $K^\alpha$-translators   with timelike axis.

\begin{theorem}\label{clas-1} Any rotational $K^\alpha$-translator   with timelike axis parametrizes as \eqref{p1}, where
\begin{equation}\label{p-rot1}
f(r) =\left\{
\begin{array}{lll}
\pm \int^{r}\left( 1- \frac{1}{m^2}e^{t^2}\right) ^{1/2}\, dt, m>1,& &\alpha =\frac{1%
}{2} \\
\pm \int^{r}\left( 1- \left( m-\frac{2\alpha-1 }{2\alpha }t^2\right) ^{\frac{%
2\alpha }{1-2\alpha }}\right) ^{1/2}\, dt,m\in \r, & &\alpha \neq \frac{%
1}{2}.%
\end{array}%
\right.
\end{equation}
The maximal  domain of the above function $f(r)$ is:
\begin{enumerate}
\item Case $\alpha=1/2$: $(0,\sqrt{\log m^2})$, where $m>1$.
\item  Case $\alpha\in (0,1/2)$: $(\sqrt{\frac{2\alpha}{2\alpha-1}m},\sqrt{\frac{2\alpha}{1-2\alpha}(1-m)})$ if $m<0$ or $(0,\sqrt{\frac{2\alpha}{1-2\alpha}(1-m)})$ if $0\leq m<1$.
 \item  Case $\alpha\not\in [0,1/2]$: $(0,\sqrt{\frac{2\alpha}{2\alpha-1}(m-1)})$ and $m>1$.
\end{enumerate}
\end{theorem}

Let us observe that in all cases, the function $f=f(r)$ is defined in a bounded domain of $\r$. It is important to know the behaviour of $f'$ at the endpoints of the domain because it explains what is happening.   So, for any $\alpha \neq 0$, we have $$ \lim_{r \to r_0} f'(r)=0, $$ where $r_0$ is the right endpoint of the domain. Then, as independently from the value of $\alpha$, the parenthesis in \eqref{k2} has a finite limit at $r_0$ because the right hand-side has $1$ as limit. Thus the second derivative of $f$ blows up at $r_0$. This implies that we cannot extend the solution beyond this endpoint. In contrast, in the particular case $\alpha\in (0,1/2)$ and $m<0$, and if $r_0=\sqrt{\frac{2\alpha}{2\alpha-1}m}$, we have
$$ \lim_{r\to r_0}f'(r)^2=1,$$
which means that at $r_0$ the surface leaves to be spacelike. As we will see in Section \ref{time},  we can extend the surface beyond of the value $r_0$ by means of a rotational timelike $K^\alpha$-translator with the same rotation axis.

The case $\alpha=1/4$ will be a distinguished case in this paper, independently of the causal character of the surface as well as the rotation axis. This is due to that equation \eqref{k2} simplifies enormously  because the terms involving $1-f'^2$ disappear and consequently, the integration can be done explicitly.  

 \begin{corollary}
Rotational   $K^{1/4}$-translators with   timelike axis parametrize as  \eqref{p1}, where
$$f(r)= \pm \frac{1}{2} \left(r \sqrt{1-m-r^2}-(m-1) \tan ^{-1}\left(\frac{r}{\sqrt{1-m-r^2}}\right)\right),$$
where $r\in (0,\sqrt{1-m})$ and $m<1$.  
\end{corollary}

We study the behaviour  of the  rotational $K^{\alpha}$-translators   that intersect the rotation axis. This implies that the generating curve $r\mapsto (r,0,f(r))$ is defined at the limit at $r=0$. Here we are interested in two situations whose intersection occurs orthogonally or  at a conical point, that is, a point where the metric is degenerated. This is equivalent 
to   $\lim_{r\to 0}f'(r)=0$ (or  $\lim_{r\to 0}h(r)=1$) in the first case  or  $\lim_{r\to 0}f'(r)^2=1$ ($\lim_{r\to 0}h(r)=\infty$) in the second one.

\begin{corollary} \label{exist-1}
There are not rotational $K^{\alpha}$-translators with timelike axis  intersecting orthogonally  the rotation axis. On the other hand, if $\alpha\in (0,1/2)$, there are rotational $K^{\alpha}$-translators with timelike axis    intersecting the rotation axis at a conical point.
\end{corollary}

\begin{proof}
If $\alpha=1/2$, $\lim_{r\to 0}h(r)=m>1$.  If $\alpha\in (0,1/2)$, then the generating curve arrives to the rotation axis when $m\in [0,1)$. In such a case  $\lim_{r\to 0}h(r)=m^{\alpha/(2\alpha-1)}\not=1$. Analogously, if $\alpha\not\in [0,1/2]$, $\lim_{r\to 0}h(r)=m^{\alpha/(2\alpha-1)}\not=1$. This proves that the intersection with the rotation axis is not orthogonal. 

The above limits also prove the non existence of conical points when $\alpha=1/2$ and $\alpha\not\in [0,1/2]$. If  $\alpha\in (0,1/2)$ and $0\leq m<1$,   the function $f$ is defined at the limit at $r=0$ with   
$$\lim_{r\to 0}h(r)=\lim_{r\to 0}m^{\alpha/(2\alpha-1)}.$$
This limit is $\infty$ if $m=0$.
\end{proof} 

This corollary contrasts with the Euclidean setting. In \cite{al}, the authors have proved the existence of rotational  $K^\alpha$-translators intersecting orthogonally the rotation axis for all values of $\alpha$. 

\begin{remark} 
For some particular values of $\alpha $ and the constant $m$,  we can explicitly write the function $f(r)$ in \eqref{p-rot1}.  Let us denote by $f_\alpha$ the solution of \eqref{p-rot1} to indicate the value of $\alpha$.  
\begin{enumerate}
\item Case   $\alpha =1$.  If $m=1$, then 
\begin{equation*}
f_1 (r)=\pm \sqrt{ r^2-4} \mp  \sqrt{2} \tan ^{-1}\left( \frac{\sqrt{ r^2-4}}{\sqrt{2}} \right).
\end{equation*}
 
\item Case that  $\alpha =1/6$.  If $m=0$, then   
\begin{equation*}
f_{\frac16}(r)=\pm \frac{\sqrt{2}}{3}(1-\sqrt{2}r)^{3/2}.
\end{equation*}
\item Case that $\alpha =1/10$.  If  we choose $m=0$ then  
\begin{equation*}
f_{\frac {1}{10}}(r)=\pm \frac{2}{15} \sqrt{1-\sqrt{2 r}} \left(-6 r+\sqrt{2 r}+2\right).
\end{equation*}
\end{enumerate}
\end{remark}
 \subsection{The axis is spacelike }  \label{ss2}
As in the case of timelike axis, we  first prove that the rotation axis is parallel to the speed  $\vec{v}$.  

\begin{proposition} \label{p-spa1}
If $\Sigma $ is a rotational $K^{\alpha }$-translator   with   spacelike axis, then $\vec{v}$ is parallel to its
rotation axis.
\end{proposition}

\begin{proof}
After a rigid motion of $\r_1^3$, we suppose that the rotation axis is the $x$-axis.  Then, the generating curve of the surface of revolution $\Sigma$ is a planar curve included in the $xz$-plane and the parametrization of the surface is \eqref{p21} where $f'^2>1$. Recall that the parametrization  \eqref{p22} is not possible, see Section \ref{sec2}.  The Gauss curvature is
$$K=-\frac{f'f''}{r(f'^2 -1)^2}$$
and the unit normal vector field
\begin{equation}\label{n-spacelike}
N=-\frac{1}{\sqrt{f'^2 -1}}(1,f'\sinh\theta,f'\cosh\theta).
\end{equation}
The equation \eqref{k1}  is now
\begin{equation*}
P_0(r)+P_1(r) \cosh \theta +P_2(r) \sinh \theta =0,
\end{equation*}%
where
\begin{eqnarray*}
P_0 &=&\left( -\frac{f'f''}{r(f'^2 -1)^2}%
\right) ^{\alpha }+v_1 (f'^2 -1)^{-1/2}=0, \\
P_1 &=&v_2 f' \sinh \theta (f'^2 -1)^{-1/2},\\
P_{2}&=&-v_3 f' \cosh \theta (f'^2 -1)^{-1/2}.
\end{eqnarray*}%
Again we deduce   $v_2=v_3=0$, obtaining the result.
\end{proof}

By this proposition, the speed $\vec{v}$ is $(v_1,0,0)$. Since we require that $\langle N,\vec{v}\rangle$ must be positive in \eqref{k1}, and taking into account the expression of $N$ in 
\eqref{n-spacelike},  we can assume $\vec{v}=(-1,0,0)$.  Now \eqref{k1} is equivalent to $P_0=0$, yielding
\begin{equation} \label{k3}
\left (-\frac{f'f''}{r(f'^2 -1)^2} \right )^\alpha = \frac{1}{\sqrt{f'^2 -1}}.
\end{equation}
Let $h=(f'^2 -1)^{-1/2}$. Therefore \eqref{k3} is 
$$ h^{\frac{\alpha -1}{\alpha}}h'= r,$$
 and integrating, we have 
\begin{equation*}
h(r)=\left\{
\begin{array}{lll}
me^{r^2/2},m>0, &  &\alpha =\frac{1}{2} \\
 \left( m+\frac{2\alpha-1 }{2\alpha }r^{2}\right) ^{\frac{\alpha }{2\alpha
-1}},m\in \r, & &\alpha \neq \frac{1}{2}.%
\end{array}%
\right.
\end{equation*}
In terms of $f'$, we write
\begin{equation}\label{deriv-s}
f'(r)=\left\{
\begin{array}{ll}
\pm\left(1+\frac{1}{m^2}e^{-r^2}\right)^{1/2},  & \alpha =\frac{1}{2} \\
\pm\left(1+\left(m+\frac{2 \alpha -1}{2 \alpha }r^2\right)^{\frac{2 \alpha }{1-2 \alpha }}\right)^{1/2},& \alpha \neq \frac{1}{2}.%
\end{array}%
\right.
\end{equation}
From the expression of $f'$, we have the next classification of the rotational $K^\alpha$-translators with spacelike axis and the domain of the function $f$. 

\begin{theorem}\label{clas-2} Any rotational $K^\alpha$-translator   with spacelike axis parametrizes as \eqref{p21}, where  
\begin{equation}\label{p-rot6}
f(r) =\left\{
\begin{array}{lll}
\pm \int^{r}\left(1+\frac{1}{m^2}e^{-t^2}\right)^{1/2}\, dt, m>0,& &\alpha =\frac{1%
}{2} \\
\pm \int^{r}\left(1+\left(m+\frac{2 \alpha -1}{2 \alpha }t^2\right)^{\frac{2 \alpha }{1-2 \alpha }}\right)^{1/2}\, dt,m\in \r, & &\alpha \neq \frac{%
1}{2}.%
\end{array}%
\right.
\end{equation}
The maximal  domain of the above function $f(r)$ is:
\begin{enumerate}
\item Case $\alpha=1/2$. The domain is $(0,\infty)$.
\item Case $\alpha\in (0,1/2)$. The domain is $ (0,\sqrt{\frac{2\alpha}{1-2\alpha}m})$, where $m>0$.
\item  Case $\alpha\not\in [0,1/2]$. The domain is $(0,\infty )$  if $m\geq 0$  or $(\sqrt{\frac{2\alpha}{1-2\alpha}m},\infty)$ if $m<0$.
\end{enumerate}
\end{theorem}

We point out  that  if  $\alpha\in (0,1/2)$, we have
$$\lim_{r\to \sqrt{\frac{2\alpha}{1-2\alpha}m}}f'(r)^2=1,$$
and in Section \ref{time}, this solution will be extended beyond of this value by means of a rotational timelike $K^\alpha$-translator   with the same rotation axis. We now study if the rotational surface can orthogonally intersect the rotation axis. 

\begin{corollary} \label{exist-12}
There are not rotational $K^{\alpha}$-translators with spacelike axis and  intersecting orthogonally  the rotation axis.  
\end{corollary}

\begin{proof} The orthogonality condition with the rotation axis is equivalent to $\lim_{r\to 0}f'(r)=0$. From \eqref{deriv-s}, we have
$$\lim_{r\to 0}f'(r)=\left\{\begin{array}{ll}
\pm\left(1+\frac{1}{m^2}\right)^{1/2},  & \alpha =\frac{1}{2} \\
\pm\left(1+ m^{\frac{2 \alpha }{1-2 \alpha }}\right)^{1/2},& \alpha \neq \frac{1}{2},%
\end{array}%
\right.$$
and in both cases, this limit cannot be $0$.
\end{proof}

The particular value $\alpha =1/4$  can be solved explicitly.

\begin{corollary}
Rotational   $K^{1/4}$-translators with    spacelike axis parametrize as  \eqref{p21}, where
$$f(r)= \pm \frac{1}{2} \left(r \sqrt{1+m-r^2}+(1+m) \tan ^{-1}\left(\frac{r}{\sqrt{1+m-r^2}}\right)\right) ,$$
where $r \in (0,\sqrt{1+m})$ and $m>-1$.
\end{corollary}

 \subsection{The axis is lightlike }  \label{ss3}

We assume that the causal character of the rotation axis   is lightlike and we prove that the rotation axis must be parallel to the speed.

\begin{proposition} \label{p-light}
If $\Sigma$ is a rotational $K^\alpha$-translator with lightlike axis, then $\vec{v}$ is parallel to the rotation axis.  
\end{proposition}

\begin{proof} Up to a rigid motion of $\r_1^3$, we can choose the axis to be $L=\mbox{sp}\{(0,1,1)\}$ and the parametrization of the surface is given by \eqref{p3} with $f'>0$. The Gauss curvature is 
$$K=\frac{f''}{8r f'^2}$$
and the unit normal vector field is 
$$N=\frac{1}{2\sqrt{f'}}(-2t,-1+t^2+f',1+t^2+f').$$
Thus, if $\vec{v}=(v_1,v_2,v_3)$,  the equation  \eqref{k1} is
\begin{equation*}
\left(\frac{f''}{8r f'^2}\right)^\alpha=\frac{1}{2\sqrt{f'}}(-2tv_1+(-1+t^2+f')v_2-(1+t^2+f')v_3 ).
\end{equation*}
This is a polynomial equation of degree $2$ on the variable $t$ and the coefficients are functions on $r$, in particular, these coefficients must vanish. The coefficient of $t^2$ yields $v_2=v_3$ and of $t$, $v_1=0$, proving the result.
\end{proof}

 Once we have proved this proposition, the right-hand side of \eqref{k1} is $\langle N,\vec{v}\rangle=-v_2/\sqrt{f'}$, in particular, $v_2<0$. Without loss of generality, we assume that $\vec{v}=(0,-1,-1)$.  Then \eqref{k1} is
$$\left(\frac{f''}{8r f'^2}\right)^\alpha=\frac{1}{\sqrt{f'}},$$
or equivalently, 
\begin{equation*} 
\frac{f''}{f'^{\frac{4\alpha-1}{2\alpha}}}=8r.
\end{equation*}
Integrating,
\begin{equation} \label{der-l}
f'(r)=\left\{
\begin{array}{ll}
 me^{4r^2},m>0,  & \alpha =\frac{1}{2} \\
\left (\frac{2(1-2\alpha)}{\alpha}r^2+m \right )^{\frac{2\alpha}{1-2\alpha}}, m\in\r ,& \alpha \neq \frac{1}{2}.%
\end{array}%
\right.
\end{equation}%
From \eqref{der-l}, we find the domain of $f$. On the other hand, the spacelike condition $f'(r)>0$ leaves to be satisfied when $f'(r)=0$. This occurs in the case $\alpha\in (0,1/2)$ and $m<0$. Indeed, we have 
$$\lim_{r\to \sqrt{\frac{\alpha}{2(2\alpha -1)}m}}f'(r)=0.$$
This solution can be extended to the left of the domain by means of a rotational timelike $K^\alpha$-translator with the same rotation axis.

\begin{theorem}\label{clas-3}  Any rotational $K^\alpha$-translator   with lightlike axis parametrizes as \eqref{p3}, where 
\begin{equation}\label{int-light}
f(r) =\left\{
\begin{array}{lll}
 m\int^{r} e^{4t^2} dt, m>0,& &\alpha =\frac{1%
}{2} \\
 \int^{r} \left (\frac{2(1-2\alpha)}{\alpha}t^2+m \right )^{\frac{2\alpha}{1-2\alpha}} dt,m\in \r, & &\alpha \neq \frac{%
1}{2}.%
\end{array}%
\right.
\end{equation}
The maximal  domain of the above function $f(r)$ is: 
\begin{enumerate}
\item Case $\alpha=1/2$. The domain is $(0, \infty)$.
\item Case $\alpha\in (0,1/2)$. The domain is $(0,\infty)$  if $m\geq 0$ or $(\sqrt{\frac{\alpha}{2(2\alpha -1)}m}, \infty)$  if $m<0$. 
\item  Case $\alpha\not\in [0,1/2]$. The domain is $(0,\sqrt{\frac{\alpha}{2(2\alpha-1)}m})$ where $m>0$.
\end{enumerate}
\end{theorem}

\begin{corollary} \label{cor-light}
Rotational   $K^{1/4}$-translators with lightlike axis parametrize as  \eqref{p3}, where
$$f(r)=\frac{4}{3}r^3+mr,$$
where $r \in (0,\infty)$ if $m >  0$ and $r \in (\sqrt{-m}/2,\infty)$ otherwise.
\end{corollary}

Again, for some   particular values of $\alpha$ and the constant $m$,  we can integrate  \eqref{int-light}. For example,  
$$f_1(r)=\pm  \frac{\sqrt{2} \tanh ^{-1}\left(\frac{\sqrt{2} r}{\sqrt{m}}\right)}{4m^{3/2}}+\frac{ r}{2m^2-4 m r^2},$$
and for any $\alpha \in (0,1/2)$ with   $m=0$,  
$$ f_{\alpha}(r)=\frac{1-2\alpha}{1+2\alpha}\left ( \frac{2(1-2\alpha)}{\alpha} \right )^{\frac{2\alpha}{1-2\alpha}} r^{\frac{1+2\alpha}{1-2\alpha}}.$$

\section{Rotational timelike $K^\alpha$-translators} \label{time}

In this section, we study the rotational timelike surfaces of $\r_1^3$ that satisfy the equation \eqref{k1}. Our first result is the relation between the speed vector $\vec{v}$ and the rotational axis.

\begin{proposition} \label{p-timelike}
If  $\Sigma $ is a rotational timelike $K^{\alpha }$-translator, then $\vec{v}$ is parallel to the rotation axis.  
\end{proposition}
\begin{proof}
The proof follows the same steps as Propositions \ref{p-tim1}, \ref{p-spa1} and \ref{p-light} because the expression of $K$ is the same that in the spacelike case with the  difference that now $1-f'^2<0$ in \eqref{p1}, $1-f'^2>0$ in  \eqref{p21} and $f'<0$ in \eqref{p3}. In case that the parametrization of the surface is \eqref{p22}, the argument is analogous.
\end{proof}

We now give the classification of the rotational timelike $K^\alpha$-translators. The proofs are similar to those of  Section \ref{spac}, so our previous investigations may be adapted in the present case.  We will omit the details.

\subsection{The axis is timelike}

We suppose that the speed is $\vec{v}=(0,0,1)$.  Now $f'^2>1$ and equation  \eqref{k1} is 
$$\left(-\frac{f'f''}{r(f'^2-1)^2}\right)^\alpha=\frac{1}{\sqrt{f'^2-1}}.$$
Using the auxiliary function  $h=1/\sqrt{f'^2-1}$ and integrating, we have 
 \begin{equation}  \label{fts}
f'(r)=\left\{
\begin{array}{ll}
\pm\left(1+\frac{1}{m^2}e^{-r^2 }\right)^{1/2},  & \alpha =\frac{1}{2} \\
\pm\left(1+\left(m+\frac{2 \alpha -1}{2 \alpha }r^2\right)^{\frac{2 \alpha }{1-2 \alpha }}\right)^{1/2},& \alpha \neq \frac{1}{2}.%
\end{array}%
\right.
\end{equation}%

\begin{theorem} \label{clas-4}
 Any rotational timelike $K^\alpha$-translator   with timelike axis parametrizes as \eqref{p1}, where
\begin{equation} 
f(r) =\left\{
\begin{array}{lll}
\pm \int^{r}\left( 1+ \frac{1}{m^2}e^{-t^2}\right) ^{1/2}\, dt, m>0,& &\alpha =\frac{1%
}{2} \\
\pm \int^{r}\left( 1+ \left( m+\frac{2\alpha-1 }{2\alpha }t^2\right) ^{\frac{%
2\alpha }{1-2\alpha }}\right) ^{1/2}\, dt,m\in \r, & &\alpha \neq \frac{%
1}{2}.%
\end{array}%
\right.
\end{equation}
The maximal domain of $f(r)$ is:
\begin{enumerate}
\item Case $\alpha=1/2$. The domain is $(0,\infty)$.  
\item Case $\alpha\in (0,1/2)$. Then $m>0$ and the domain is $(0,\sqrt{\frac{2\alpha}{1-2\alpha}m}).$
\item Case $\alpha\not\in [0,\frac12]$. The domain is $(0,\infty)$ if $m\geq 0$ or $(\sqrt{\frac{2\alpha}{1-2\alpha}m},\infty)$ if $m < 0$.
 \end{enumerate}
\end{theorem}

Such as it was announced in Section \ref{spac}, we can extend some rotational $K^\alpha$-translators with timelike axis by means of timelike $K^\alpha$-translators with the same axis. This occurs in the case $\alpha\in (0,1/2)$ of Theorem \ref{clas-4}  when $m>0$. At the endpoint $r_0=\sqrt{\frac{2\alpha}{1-2\alpha}m}$, we have $\lim_{r\to r_0}f'(r)^2=1$. Then we can extend $f$ to the right of  the value $r_0$ with the rotational $K^\alpha$-translator with the same rotation axis that appeared in Theorem \ref{clas-1} but with reverse sign of $m$.

Using the same argument as in Corollary \ref{exist-12}, and using \eqref{fts}, we have: 

\begin{corollary}  
There are not rotational timelike $K^{\alpha}$-translators with timelike axis and  intersecting orthogonally  the rotation axis.  
\end{corollary}

\begin{corollary}
Rotational  timelike  $K^{1/4}$-translators with   timelike axis parametrize as  \eqref{p1}, where
$$f(r)= \pm \frac{1}{2} \left(r \sqrt{1+m-r^2}+(1+m) \tan ^{-1}\left(\frac{r}{\sqrt{1+m-r^2}}\right)\right),$$
with $r\in (0,\sqrt{1+m})$ and $m>-1$.  
\end{corollary}

\subsection{The axis is spacelike}

We assume that the rotation axis is $(1,0,0)$. We separate the two parametrizations. First, consider the parametrization \eqref{p21} with  $f'^2<1$. As a consequence of Proposition \ref{p-timelike},  we can assume $\vec{v}=(-1,0,0)$. Equation \eqref{k1} is
\begin{equation*} 
\left (-\frac{f'f''}{r(1-f'^2)^2} \right )^\alpha = \frac{1}{\sqrt{1-f'^2 }}.
\end{equation*}
 Using    $h=(1-f'^2 )^{-1/2}$ and integrating, we obtain  
\begin{equation}\label{fts2}
f'(r)=\left\{
\begin{array}{ll}
\pm\left(1-\frac{1}{m^2}e^{r^2}\right)^{1/2},  & \alpha =\frac{1}{2} \\
\pm\left(1-\left(m-\frac{2 \alpha -1}{2 \alpha }r^2\right)^{\frac{2 \alpha }{1-2 \alpha }}\right)^{1/2},& \alpha \neq \frac{1}{2}.%
\end{array}%
\right.
\end{equation}

\begin{theorem} \label{clas-5}
For any rotational  timelike $K^\alpha$-translator with spacelike axis and parametrized  by \eqref{p21}, the function $f$ is  
\begin{equation*}
f(r) =\left\{
\begin{array}{lll}
\pm \int^{r}\left(1-\frac{1}{m^2}e^{t^2}\right)^{1/2}\, dt, m>0,& &\alpha =\frac{1%
}{2} \\
\pm \int^{r}\left(1-\left(m-\frac{2 \alpha -1}{2 \alpha }t^2\right)^{\frac{2 \alpha }{1-2 \alpha }}\right)^{1/2}\, dt,m\in \r, & &\alpha \neq \frac{%
1}{2}.%
\end{array}%
\right.
\end{equation*}
The maximal  domain of the above function $f(r)$ is:
\begin{enumerate}
\item Case $\alpha=1/2$. The domain is $ (0,\sqrt{\log m^2})$.
\item Case $\alpha\in (0,1/2)$. The domain is $ (0,\sqrt{\frac{2\alpha}{1-2\alpha}(1-m)})$ if $ 0 \leq m <1$ and $(\sqrt{\frac{2\alpha}{2\alpha -1}m},\sqrt{\frac{2\alpha}{1-2\alpha}(1-m)} )$ if $m<0$.
\item  Case $\alpha\not\in [0,1/2]$. The domain is $(0,\sqrt{\frac{2\alpha}{1-2\alpha }(1-m)}) $, where $m>1$.
\end{enumerate}
\end{theorem}

When $\alpha\in (0,1/2)$ and $m<0$, we have 
$$\lim_{r\to \sqrt{\frac{2\alpha}{2\alpha-1}m}}f'(r)^2=1.$$
This property connects with the analogous one of the   rotational $K^\alpha$-translators with the same axis of Theorem \ref{clas-2}, where the integration constant $m$ is positive. Therefore, each surface is the extension of the other one and the function $f(r)$ is defined in the interval $(0,\sqrt{\frac{2\alpha}{1-2\alpha}(1-m)})$. 

\begin{corollary} \label{exist-2}
There are not rotational  timelike $K^\alpha$-translators with spacelike axis intersecting orthogonally the rotation axis. If $\alpha\in (0,1/2)$, there are rotational timelike $K^{\alpha}$-translators with spacelike axis that meet the rotation axis at a conical point.  
\end{corollary}
\begin{proof} The orthogonality condition is equivalent to $\lim_{r\to 0}f'(r)=0$ and the intersection is at a conical point if $\lim_{r\to 0}f'(r)^2=1$. From \eqref{fts2},
$$\lim_{r\to 0} f'(r)=\left\{
\begin{array}{ll}
\pm\left(1-\frac{1}{m^2}\right)^{1/2},  & \alpha =\frac{1}{2} \\
\pm\left(1- m^{\frac{2 \alpha }{1-2 \alpha }}\right)^{1/2},& \alpha \neq \frac{1}{2}.%
\end{array}%
\right.$$
Thus   the intersection is orthogonal if $m=1$, but this is not possible by Theorem \ref{clas-5}. The limit is $1$ only if $m=0$ and $\alpha\in (0,1/2)$, obtaining a intersection of conical type.
\end{proof}

\begin{corollary}
For rotational   timelike $K^{1/4}$-translators with spacelike axis parametrized by  \eqref{p21}, we have
$$f(r)= \pm \frac{1}{2} \left(r \sqrt{1-m-r^2}+(1-m) \tan^{-1} \left(\frac{r}{\sqrt{1-m-r^2}}\right)\right) ,$$
where $r \in (0,\sqrt{1-m})$ with $m <1$.
\end{corollary}

Consider now  the parameterization \eqref{p22}. By Proposition \ref{p-timelike}, we can assume that $\vec{v}=(1,0,0)$. Then equation \eqref{k1} is
\begin{equation*} 
\left (\frac{f'f''}{r(1+f'^2)^2} \right )^\alpha = \frac{1}{\sqrt{1+f'^2 }}.
\end{equation*}
Using the function $h=(1+f'^2)^{-1/2}$, we obtain the expression of $f'$, namely, 
\begin{equation} \label{der-p22}
f'(r)=\left\{
\begin{array}{ll}
\pm\left(\frac{1}{m^2}e^{r^2} -1 \right)^{1/2},  & \alpha =\frac{1}{2} \\
\pm\left(\left(m-\frac{2 \alpha -1}{2 \alpha }r^2\right)^{\frac{2 \alpha }{1-2 \alpha }} -1 \right)^{1/2},& \alpha \neq \frac{1}{2}.%
\end{array}%
\right.
\end{equation}

\begin{theorem}
For any rotational  timelike $K^\alpha$-translator with spacelike axis and parametrized  by \eqref{p22}, the function $f$ is
\begin{equation*} 
f(r) =\left\{
\begin{array}{lll}
\pm \int^{r}\left(\frac{1}{m^2}e^{t^2} -1\right)^{1/2}\, dt, m>0,& &\alpha =\frac{1%
}{2} \\
\pm \int^{r}\left(\left(m-\frac{2 \alpha -1}{2 \alpha }t^2 \right)^{\frac{2 \alpha }{1-2 \alpha }} -1\right)^{1/2}\, dt,m\in \r, & &\alpha \neq \frac{%
1}{2}.%
\end{array}%
\right.
\end{equation*}
The maximal  domain of the above function $f(r)$ is:
\begin{enumerate}
\item Case $\alpha=1/2$. The domain is $(0,\infty )$ if $m \leq 1$ or $(\sqrt{\log m^2}, \infty)$ if $m>1$.
\item Case $\alpha\in (0,1/2)$. The domain is $ (0,\infty )$ if $m \geq 1$ or $ (\sqrt{\frac{2\alpha}{1-2\alpha}(1-m)}, \infty)$ if $m<1$.
\item  Case $\alpha\not\in [0,1/2]$. The domain is $(0,\sqrt{\frac{2\alpha}{2\alpha -1}m} ) $ if $0 <m \leq 1$ or $(\sqrt{\frac{2\alpha}{2\alpha -1}(m-1)}, \sqrt{\frac{2\alpha}{2\alpha -1}m}) $ if $m>1$.
\end{enumerate}
\end{theorem}

\begin{corollary} For a rotational timelike   $K^{1/4}$-translators with    spacelike axis parametrized by  \eqref{p22}, we have 
$$f(r)= \pm \frac{1}{2} \left(r \sqrt{-1+m+r^2}+(-1+m) \log \left(r+{\sqrt{-1+m+r^2}}\right)\right) ,$$
where $r \in (\sqrt{1-m}, \infty)$ with $m \leq 1$. If $m=1$ then $r \in (0, \infty)$  and the generating curve is the parabola $r \mapsto (\pm \frac{1}{2} r^2,r,0)$.
\end{corollary}

In the following we see when the intersection with the rotation axis is orthogonal.  
\begin{corollary} \label{exist-3}
 For any $\alpha$, there are rotational  timelike $K^{\alpha}$-translators with spacelike axis parametrized by \eqref{p22} which intersect orthogonally  the rotation axis.  
\end{corollary}
\begin{proof}
If $\alpha \in (0,1/2]$, it is necessary to be $m=1$ in \eqref{der-p22} where the maximal domain is $(0,\infty)$. The same occurs for $\alpha \notin [0,1/2]$ and $m=1$ again and the domain is  $(0,\sqrt{\frac{2\alpha}{2\alpha -1}})$. The limits follows from  \eqref{der-p22}.
\end{proof}

\subsection{The axis is lightlike}

 We assume that  the rotation axis is $L=\mbox{sp}\{ e_2 +e_3 \}$, $\vec{v}=(0,-1,-1)$ and the parametrization is \eqref{p3}  where $f'<0$. Equation  \eqref{k1} is
\begin{equation*} 
\frac{f''}{(-f')^{\frac{4\alpha-1}{2\alpha}}}=8r.
\end{equation*}
Integrating,
\begin{equation*}
f'(r)=\left\{
\begin{array}{ll}
 -me^{-4r^2},m>0,  & \alpha =\frac{1}{2} \\
-\left (\frac{2(2\alpha-1)}{\alpha}r^2+m \right )^{\frac{2\alpha}{1-2\alpha}}, m\in\r ,& \alpha \neq \frac{1}{2}.%
\end{array}%
\right.
\end{equation*}%

\begin{theorem}\label{clas-7}  Any rotational timelike $K^\alpha$-translator   with lightlike axis parametrizes as \eqref{p3}, where 
\begin{equation*}
f(r) =\left\{
\begin{array}{lll}
 -m\int^{r} e^{-4t^2} dt, m>0,& &\alpha =\frac{1%
}{2} \\
- \int^{r} \left (\frac{2(2\alpha-1)}{\alpha}t^2+m \right )^{\frac{2\alpha}{1-2\alpha}} dt,m\in \r, & &\alpha \neq \frac{%
1}{2}.%
\end{array}%
\right.
\end{equation*}
The maximal  domain of the above function $f(r)$ is: 
\begin{enumerate}
\item Case $\alpha=1/2$. The domain is $(0, \infty)$.
\item Case $\alpha \in (0,1/2)$. The domain is $(0, \sqrt{\frac{\alpha}{2(1-2\alpha )}m} )$ and $m>0$.
\item Case $\alpha \not\in [0,1/2]$. The domain is $(0,\infty)$ if $m\geq 0$ and $(\sqrt{\frac{\alpha}{2(1-2\alpha )}m}, \infty)$ if $m<0$.
\end{enumerate}
\end{theorem}

 Following with the problem of extension of rotational $K^\alpha$-translators along the limit set of lightlike points, consider the case $\alpha\in (0,1/2)$ and $m>0$ of Theorem \ref{clas-7}. If $r_0=\sqrt{\frac{\alpha}{2(1-2\alpha )}m}$, we have $\lim_{r\to  r_0}f'(r)=0$. This implies that at $r_0$ the surface leaves to be timelike and the metric is degenerated. From Theorem \ref{clas-3} and in the case $\alpha\in (0,1/2)$, $m<0$, the generating curve of the corresponding rotational $K^\alpha$-translator is defined in the interval $( \sqrt{\frac{\alpha}{2(2\alpha-1 )}m},\infty)$. Thus the rotational timelike $K^\alpha$-translator can be extended on the right of $r_0$ by means of a  rotational $K^\alpha$-translator with the same axis. In this situation, the function $f(r)$ is defined in $(0,\infty)$.

\begin{corollary} \label{c-light}
Rotational   timelike $K^{1/4}$-translators with    lightlike axis parametrize as  \eqref{p3}, where
$f(r)=\frac{4}{3}r^3-mr$, $m>0$ and   $r\in (0,\frac{\sqrt{m}}{2})$.
\end{corollary}


\section*{Acknowledgements}
Rafael L\'opez  is a member of the Institut of Mathematics  of the University of Granada. This work  has been partially supported by  the Projects  I+D+i PID2020-117868GB-I00, supported by MCIN/ AEI/10.13039/501100011033/,  A-FQM-139-UGR18 and P18-FR-4049.


 \end{document}